\documentclass{amsart}

\usepackage[dvipsnames]{xcolor}
\usepackage{tikz}
\usepackage{amsmath,bm,bbm,amsthm, amssymb}
\usepackage{mathtools}
\usepackage{fullpage}
\usepackage{color}
\usepackage{dsfont}
\usepackage{animate}
\usepackage{etoolbox}
\usepackage{comment}
\usepackage{pgf}
\usetikzlibrary{calc}
\usetikzlibrary{patterns}
\usetikzlibrary{arrows}
\usetikzlibrary{decorations.pathreplacing}
\usepackage[utf8]{inputenc}
\usepackage{pgfplots}

\usepackage{tcolorbox}
\usepackage{adjustbox}
\usepackage[final]{pdfpages}
\usepackage[ruled,vlined,linesnumbered,algosection,resetcount]{algorithm2e}
\usepackage{blkarray}
\usepackage{arydshln}
\usepackage{wrapfig}

\usepackage[textwidth=2cm]{todonotes}
\allowdisplaybreaks
\theoremstyle{plain}
\newtheorem{theorem}{Theorem}
\newtheorem{proposition}[theorem]{Proposition}
\newtheorem{corollary}[theorem]{Corollary}
\newtheorem{lemma}[theorem]{Lemma}

\theoremstyle{definition}

\theoremstyle{remark}
\newtheorem{remark}[theorem]{Remark}

\def\ov{\overline}

\def\R{\mathbb R}

\def\P{\mathbb P}

\def\E{\mathbb E}

\def\eps{\varepsilon}

\def\00{\mathbf 0}

\def\bet{\begin{theorem}}
\def\ent{\end{theorem}}
\def\bec{\begin{corollary}}
\def\enc{\end{corollary}}
\def\bep{\begin{proof}}
\def\enp{\end{proof}}
\def\f{\frac}
\def\om{\omega}

\def\g{\gamma}
\def\la{\lambda}
\def\es{\emptyset}

\def\ms{\mathsf}

\def\N{\mathbb N}

\def\one{\mathds1}

\def\d{\mathrm d}

\renewcommand\leq{\leqslant}
\def\r{\rho}

\renewcommand\le{\leqslant}
\renewcommand\ge{\geqslant}

\def\bel{\begin{lemma}}
\def\enl{\end{lemma}}

\def\been{\begin{enumerate}}
\def\enen{\end{enumerate}}

\def\sm{\setminus}

\def\k_d{\kappa_d}

\def\s{\sigma}

\def\k{\kappa}

\def\bepr{\begin{proposition}}
\def\enpr{\end{proposition}}
\def\vp{\varphi}

\def\Cov{\ms{Cov}}
\def\De{\Delta}

\begin{document} 

\author{Moritz Otto}
\address[Moritz Otto]{Department of Mathematics, Aarhus University, Denmark}
\email{otto@math.au.dk}

\subjclass[2010]{Primary 60K35. Secondary 60G55, 60D05.}
\keywords{coupling method, determinantal process, Ginibre process, Kantorovich-Rubinstein distance, Palm calculus, Poisson approximation, scaling limit}
\renewcommand{\thefootnote}{\fnsymbol{footnote}}

\title{Couplings and Poisson approximation for stabilizing functionals of determinantal point processes
} 
\maketitle

\begin{abstract} 
We prove a Poisson process approximation result for stabilizing functionals of a determinantal point process. Our results use concrete couplings of determinantal processes with different Palm measures and exploit their association properties. Second, we focus on the Ginibre process and show in the asymptotic scenario of an increasing window size that the process of points with a large nearest neighbor distance converges after a suitable scaling to a Poisson point process. As a corollary, we obtain the scaling of the maximum nearest neighbor distance in the Ginibre process, which turns out to be different from its analogue for independent points.
\end{abstract}

\noindent

\vspace{0.1cm}
\noindent

\section{Introduction}\label{sintro}

Determinantal point processes (DPPs) were introduced in quantum mechanics to study configurations of fermions \cite{macchi}. Due to their repulsive nature, the play a fundamental role in applied sciences, e.g.~as a model for base stations in a wireless network \cite{miyoshi}. In mathematics, DPPs arise naturally in different fields, such as eigenvalues of random matrices \cite{diaconis} and random spanning trees \cite{burton1993}. DPPs have notable probabilistic properties. Amongst others, the (reduced) Palm process of a DPP is again a determinantal process and important quantities such as the Laplace transform and Janossy densities admit closed-form expressions (see e.g. \cite{torrisi}). A DPP on $\R^d$ is determined by its correlation kernel $K$, which is a Hermitian function from $\R^d \times \R^d$ to $\mathbb C$. An important DPP is the Ginibre process on $\R^2$ with Gaussian kernel given in Section 2. 

Let $\xi$ be a stationary DPP on $\R^d$ and let $g$ be a measurable function from $\R^d \times \mathbf N$ to $\{0,1\}$, where we write $\mathbf N$ for the set of $\sigma$
-finite point configurations on $\R^d$. For $x\in \R^d$, let $\delta_x$ denote the Dirac measure in $x$. For some measurable $W \subset \R^d$, let
$$
\Xi[\xi]:=\sum_{x \in \xi \cap W} g(x,\xi) \delta_x.
$$
Here, the function $g$ has the effect of a thinning of $\xi$, in the sense that $\Xi$ is the point process of all points $x \in \xi$ in the set $W$ which satisfy $g(x,\xi)=1$. Random measures of this type are flexible models that appear in the study of random spatial graphs, stochastic topology and geometric extreme value theory.

In this article, we study the distance (in an appropriate distance of point processes) of $\Xi$ and a Poisson point process. To the best of our knowledge, this is the first paper which systematically studies Poisson approximation for determinantal processes. This continues the studies for stabilizing functionals of Poisson point processes \cite{DST16,BSY21,O22}, Poisson hyperplanes processes \cite{O21} and Gibbs point processes \cite{LO}. However, the tools used in our paper are different then in above mentioned. As main contributions, this article shows:

(i) If the correlation kernel $K$ is fast decaying and if the thinning function $g$ is stabilizing and satisfies some natural assumptions, the bound of the distance of $\Xi$ and a Poisson process is analogous to the bounds obtained for thinned Poisson point processes (see \cite{BSY21} and \cite{O21}).

(ii) If $\xi$ is the Ginibre process and $g(x,\xi)$ is indicator which is one if $\xi\setminus \{x\}$ does not have points in a ball with a given radius $v$ around $x$, we prove in as asymptotic scenario where the size of $W$ and $v$ tend to infinity, that an appropriate scaling of $\Xi$ converges to a Poisson process.

Our paper is organized as follows. In Section \ref{sec:model} we introduce determinantal point processes state our two main results. In Section \ref{sec:prelim}, we provide important notions such as Palm theory, negative association and correlation decay. The proof of Theorem \ref{th:decr} is given in Section \ref{sec:pr1}. In Section \ref{sec:pr2} we provide the proof of Theorem \ref{th:balls}.

\section{Model and main results} \label{sec:model}

We  work on the Euclidean space $\R^d$ ($d \ge 1$) equipped with its Borel $\s$-field $\mathcal B^d$ and Euclidean norm $\|\cdot\|$. We denote by $\mathbf N$  the space of all $\sigma$-finite counting measures on $\R^d$, and by $\widehat{\mathbf{N}}$ the space of all finite counting measures on $\R^d$ and equip $\mathbf N$ and $\widehat{\mathbf{N}}$ with their corresponding $\sigma$-fields $\mathcal N$ and $\widehat{\mathcal N}$, which are induced by the maps $\omega \mapsto \omega(B)$ for all $B \in \mathcal B^d$. A {\em point process} is a random
element $\xi$ of $\mathbf N$, defined over some fixed
probability space $(\Omega,\mathcal F,\P)$.
The {\em intensity measure} of $\xi$ is the
measure $\E[\xi]$ defined by $\E[\xi](B):=\E[\xi(B)]$, $B\in\mathcal B^d$. For $z \in \R^d$ and $r>0$ let $B_r(z)$ be the closed Euclidean ball with radius $r$ around $z$. For $B \in \mathcal B^d$ we write $|B|$ for the Lebesgue measure of $B$.

Let $K:(\R^d)^2 \to \mathbb{C}$ be a complex function. We say that $\xi$ is a {\em determinantal point process with correlation kernel $K$}, if for every $n \in \N$ and pairwise disjoint $A_1,\dots,A_n \in \mathcal B^d$ we have that
\begin{align*}
	\E [\xi(A_1)\cdots \xi(A_n)]=\int_{A_1 \times \cdots \times A_n} \det (K(x_i,x_j))_{i,j=1}^n \d(x_1,\dots x_n), 
\end{align*} 
where $\d\dots$ denotes integration with respect to the Lebesgue measure on $\R^d$, $(K(x_i,x_j))_{i,j=1}^m$ is the $m\times m$-matrix with entry $K(x_i,x_j)$ at position $(i,j)$, and $\det M$ is the determinant of a complex-valued $m\times m$-matrix $M$. This says that $\xi$ has correlation functions of all orders and the $m$th order correlation function $\rho^{(m)}$ is given by
\begin{align*}
\rho^{(m)}(x_1,\dots,x_m)=	\det (K(x_i,x_j))_{i,j=1}^m,\quad x_1,\dots,x_m \in \R^d,\,\quad m \in \mathbb{N},
\end{align*}
and that it is locally integrable. In this article we assume that $K$ satisfies the following assumptions (i)-(iv).
\begin{enumerate}
	\item[(i)] $K$ is Hermitian, i.e.~$K(x,y)=\overline{K(y,x)}$, $x,y \in \R^d$,
	\item[(ii)] $K$ is locally square integrable, i.e.~for every compact $B \in \mathcal B^d$ the integral $$\int_B \int_B |K(x,y)|^2 \d y \d x$$ is finite,
	\item[(iii)] $K$ is locally of trace class, i.e.~for every compact $B \in \mathcal B^d$ the integral $\int_B K(x,x) \d x$ is finite,
\end{enumerate}
 Under the assumptions (i)--(iii), it follows from Mercer's theorem that for every compact $B \subset \R^d$, the restriction $\xi_B$ of $\xi$ to $B$ is a determinantal point process whose kernel $K_B$ is for almost all $(x,y)\in B \times B$ given by
\begin{align*}
	K_B(x,y)=\sum_{k=1}^\infty \lambda_k^B \phi_k^B(x) \overline{\phi_k^B(y)},
\end{align*}
where $\lambda_k^B \in \mathbb{R}$, $k \in \mathbb{N}$, and the functions $\phi_k^B$, $k \in \mathbb{N}$, form on orthonormal base of $L^2(B)$. Finally, we assume that
\begin{enumerate}
	\item[(iv)] $0 \le \lambda_k^B \le 1$ for all $k \in \mathbb{N}$ and all compact $B \in \mathcal B^d$.
\end{enumerate} 
Under the assumptions (i)--(iv) there is a unique (in distribution) determinantal point process with correlation kernel $K$ (see \cite[Theorem 3]{S00}).

For $x \in \R^d$ we call $\xi^x$ a Palm version of the point process $\xi$ at $x$, if for all measurable $f:\R^d \times \mathbf N \to \R_+$,
\begin{align}
\E\Big[\int f(x,\xi) \xi(\d x)\Big]=\int f(x,\xi^x) \E[\xi](\d x).\label{Palmsp}
\end{align}
Later, we will generalize this definition and define a Palm process with respect to another point process. 

Let $\xi$ be a stationary determinantal process satisfying (i)--(iv) with intensity $\r >0$. Let $g:\R^d \times \mathbf N \to \{0,1\}$ be a measurable function (called {\em score function}) and let $W \in \mathcal B^d$. We define
\begin{align}
\Xi[\om]:=\sum_{x \in \om \cap W} g(x,\om) \delta_{x}, \label{def:Xigen}
\end{align}
and set $\Xi:=\Xi[\xi]$. Note that by \eqref{Palmsp}, the intensity measure $\mathbf L$ of $\Xi$ is given by
$$
\mathbf L(A)=\r \int_{W \cap A} \E[g(x,\xi^x)]\, \d x,\quad A \in \mathcal B^d.
$$
In this article we study the Kantorovich-Rubinstein (KR) distance of $\Xi$ and a finite Poisson process. We recall the definition of the KR distance from \cite{DST16}. For finite point processes $\zeta$ and $\xi$ on $\R^d$ the KR distance is given by
\begin{align*}
	\mathbf{d_{KR}}(\zeta, \xi) := \sup_{h \in \text{Lip}}|\mathbb{E} h(\zeta)-\mathbb{E} h(\xi)|,
\end{align*}
where $\text{Lip}$ is the class of all measurable 1-Lipschitz functions $h:\widehat{\mathbf{N}} \to \mathbb{R}$ with respect to the total variation between measures $\omega_1, \omega_2$ on $\R^d$ given by
\begin{align*}
	d_{\text{TV}}(\omega_1,\omega_2):=\sup |\omega_1(A)-\omega_2(A)|,
\end{align*}
where the supremum is taken over all $A \in \mathcal B^d$ with $\omega_1(A), \omega_2(A)<\infty$. Under appropriate conditions on $\xi$ and $g$, we prove that $\Xi$ can be approximated by a Poisson process. 

We suppose that there exists $\alpha\in (0,\infty)$ such that for all $A \in \mathcal B^d$ and all $\om \in \mathbf N$, 
 \begin{align}
 \sum_{x \in \om \cap A} g(x,\om) <\alpha |A|,\label{kiss}
 \end{align}
and assume that $g$ is monotonic in the sense that for all $x \in W$, it holds that 
\begin{align}
g(x,\om_1) \le g(x,\om_2) \quad \text{or}\quad g(x,\om_1) \ge g(x,\om_2),\qquad \om_1 \subset \om_2.\label{mono}
\end{align}
We further assume that $g$ is stabilizing, by which we mean that there es a measurable function $\mathcal S:\R^d \times \mathbf N\to \mathcal F$ such that 
$$ g(x, \om) = g(x, \om \cap \mathcal S(x,\om)) 
$$  holds for any $\om \in \mathbf{N}$ and any $x \in \R^d$, where $ \mathcal{F}$ is the set of closed subsets in $\R^d$. Further suppose that $\mathcal S$ is a {\em stopping set}, which says that
$$
\{\om \in \mathbf N:\, \mathcal S(x,\om) \subset S\}=\{\om \in \mathbf N:\,\mathcal S(x,\om \cap S) \subset S\}.
$$

Moreover, we assume that the kernel $K:(\R^d)^2 \to \mathbb{C}$ satisfies
\begin{align}
	|K(x,y)|\le \phi(\|x-y\|),\quad x,y \in \R^d,\label{fastdecay}
\end{align}
for some decreasing function $\phi:\mathbb{R}_+ \to \mathbb{R}_+$ with $\lim_{r \to \infty}  \phi(r) =0$.

\begin{theorem} \label{th:decr}
	Let $\xi$ be a stationary determinantal process with kernel $K$ satisfying \eqref{fastdecay} and intensity $\rho \in (0,\infty)$. Let $\Xi$ be the score sum defined in \eqref{def:Xigen} with intensity measure $\mathbf L$ and suppose that $g$ satisfies \eqref{kiss} and \eqref{mono} and is stabilizing with respect to the stopping set $\mathcal S$. For Borel sets $S, T \subset \R^d$ with $o\in S \subset T$ let $S_x:=x+S$ and $T_x:=x+T,\,x \in W,$  and define
	$$
	\tilde g(x,\om):=g(x,\om) \one\{\mathcal S(x,\om)\subseteq S_x\},\quad x \in \R^d,\,\om \in \mathbf N.
	$$
	Let $\zeta$ be a finite Poisson process on $\R^d$ with intensity measure $\mathbf M$. Then,
	\begin{align*}
		&\mathbf{d_{KR}}(\Xi,\zeta) \le d_{TV}(\mathbf L, \mathbf M) +  2(E_1+E_2+E_3) +F,
	\end{align*}
where
\begin{align*}
	E_1&:=\r \int_W \P(\mathcal S(x,\eta^x) \not \subset S) \, \d x,\\
	E_2&:=\r^2 \int_W \int_{W \cap T_x} \E[\tilde g(x,\xi^x)] \E[\tilde g(y,\xi^y)]\,\d y \d x,\\
	E_3&:=\r^2 \int_W \int_{W \cap T_x} \E[\tilde g(x,\xi^{x,y}) \tilde g(y,\xi^{x,y})]\,\d y \d x,\\
	F&:=c\|K\| \max(|S|,1) |W\oplus S|^2 \phi(d(T,S)),
\end{align*}
where $\|K\|=\sup_{x,y \in \R^d} |K(x,y)|$, $W \oplus S$ is the Minkowski sum of $W$ and $S$, $d(T,S)$ is the Hausdorff distance of $T$ and $S$ and the constant $c>0$ does not depend on $K$, $g$, $W$, $S$ and $T$.
\end{theorem}

In the second part of this paper, we give an application of Theorem \ref{th:decr} for a concrete choice of $\xi$ and of $g$.  Let $\xi$ be the (infinite) Ginibre process, which is a stationary determinantal point process on $\mathbb C$ with correlation kernel given by \[K(z,w)= \pi^{-1}e^{-(|z|^2+|w|^2)/2} e^{z\overline{w}},\quad z,w \in \mathbb C.\]
Hence, $\xi$ has intensity $\rho=\pi^{-1}$ and it holds that $|K(z,w)|\le \phi(\|z-w\|)$ with $\phi(r):=\pi^{-1} \exp(-r^2/2)$, $r>0$, for all $z,w \in \mathbb C$.

In Theorem \ref{th:balls} below, we choose $g$ depending on $n \in \N$. Let $g_n$ is the indicator function which is one if and only if the process $\xi \setminus \{x\}$ is empty in a ball with a certain radius $v_n$ (chosen such that $v_n \to \infty$ as $n \to \infty$) around $x$. This choice leads to the study of large nearest neighbor balls. It is also an important prototype for more sophisticated models in stochastic geometry and has been studied extensively for different point processes in various spaces (see \cite{OT23}).

We consider $\xi$ as a random set in $\R^2$. Let $B_n:=B_n(o)$ the closed ball with radius $n>0$ in $\R^2$ centered at the origin $o$. We consider the process
\begin{equation}
	\Xi_n:=\sum_{x \in \xi \cap B_n} \one\{\xi(B_{v_n}(x)\setminus \{x\})=0\} \ \delta_{x}.	\label{nnxidef}
\end{equation}
as well as the scaled process
\begin{equation}
	\Psi_n:=\sum_{y \in \Xi}\delta_{y/n}=\sum_{x \in \xi \cap B_n} \one\{\xi(B_{v_n}(x)\setminus \{x\})=0\} \ \delta_{x/n}.	\label{nnpsidef}
\end{equation}



In the following theorem we compare $\Psi_n$ with a Poisson process on the unit ball $B_1$ in $\R^2$.

\begin{theorem}
	\label{th:balls}
	Let $\tau>0$ and let $\nu$ be a stationary  Poisson process on $\R^2$ with intensity $\tau>0$. There exists a sequence $(v_n)_{n \in \N}$ with $v_n^4 \sim 4\log n$ as $n \to \infty$ such that for all $n \in \N$ and any $\eps>0$,
	\[
	\mathbf{d_{KR}}(\Psi_n,\nu \cap B_1) \le Cn ^{\eps-1/16}.
	\]
\end{theorem}

As an application of the above theorem, we consider largest distances to the nearest neighbor. That is, for any $x\in \R^2$, we denote by $\mathsf{nn}(x,\xi):=\arg\min_{y\in \xi\setminus \{x\}} |x-y|$ the nearest neighbor of $x$ in $\xi\setminus \{x\}$. 

\begin{corollary}
	\label{cor:maxdistance}
We have as $n \to \infty$,
	$$
\f{1}{2\pi \sqrt{\log n}} \max_{x \in \xi \cap B_n} |B_{\mathsf{nn}(x,\xi)}| \stackrel{\P}{\longrightarrow} 1.
	$$
\end{corollary}

The proof of Corollary \ref{cor:maxdistance} is quite standard (see e.g. \cite[Corollary 4.2]{O22} or \cite[Corollary 1]{hirsch}) and therefore omitted. 

\begin{remark}
	(i)	One should compare Theorem \ref{th:decr} with \cite[Theorem 4.1]{BSY21} (or the refined version \cite[Theorem 7]{O21}) that discusses Poisson process approximation for score sums built on a Poisson process. The terms $E_1$, $E_2$, $E_3$ in Theorem \ref{th:decr} are the analogues to the terms $E_1$, $E_2$, $E_3$ in \cite[Theorem 4.1]{BSY21}. Due to the spatial independence property of the Poisson process, there is no analogue of the term $F$ (which reflects the correlation decay of the determinantal process $\xi$) in \cite{BSY21} and \cite{O21}. Note also that for a wide class of determinantal point processes (including the Ginibre process), the $\beta$-mixing coefficient does not decay exponentially fast (see \cite[Proposition 4.2]{poinas}). Therefore, the {\em exponential decay dependence (EDD)} property from \cite{beta} is violated and general results for Poisson approximation of strongly mixing processes do not apply.
	
	(ii) Note that the scaling of the maximum nearest neighbor ball in Corollary \ref{cor:maxdistance} is different from its analogue for independent points, which is proportional to $\log n$ as $n \to \infty$ (see \cite{cho}). It seems interesting to investigate whether Theorem \ref{th:balls} can be extended to $k$-nearest neighbor distances ($k \ge 2$). However, this extension would require fine probabilistic estimates on empty-space probabilities of the Ginibre process, which are beyond the scope of this article. 
\end{remark}



\section{Preliminaries} \label{sec:prelim}
\subsection{Palm calculus and  negative association}

Following \cite[Chapter 6]{K17} we next introduce Palm measures and thereby generalize the definition given in \eqref{Palmsp}. Let $\xi,\Xi$ be point processes on $(\R^d,\mathcal{B}^d)$ and assume that $\Xi$ has $\sigma$-finite intensity measure $\mathbf L$.  Then there are point processes $\xi^{x,\Xi},\,x \in \R^d$, such that
\begin{align}
	\E \Big[\int f(x,\xi)\,\Xi(\mathrm{d}x)\Big]=\int \E [f(x,\xi^{x,\Xi})]\mathbf L(\d x), \label{defPalm}
\end{align}	
where $f:\R^d \times \mathbf{N} \to [0,\infty)$ is assumed to be measurable. The processes $\xi^{x,\Xi}, x \in \R^d,$ are called {\em Palm processes} of $\xi$ with respect to $\Xi$ at $x$. If $\Xi$ is simple, $\xi^{x,\Xi}$ can be interpreted as the process $\xi$ seen from $x$ and conditioned on $\Xi$ having a point in $x$. From \cite[Lemma 6.2(ii)]{K17} it follows that $\delta_x \in \xi^{x,\Xi}$ a.s. This allows us to define the reduced Palm process $\xi^{x!,\Xi}:=\xi^{x,\Xi}-\delta_x$. If $\xi=\Xi$ a.s., we write $\xi^x$ for a Palm process of $\xi$ (with respect to itself) at $x$ (c.f.~\eqref{Palmsp}) and $\xi^{x!}$ for a reduced Palm process.

Let $\xi$ be a determinantal point process satisfying the conditions (i)--(iv) from Section \ref{sec:model} with correlation kernel $K$.  Then the reduced Palm processes $\xi^{x!},\, x \in \R^d,$ are determinantal processes with correlation kernel $K^x,\, x \in \R^d,$ given by
\begin{align}
	K^x(z,w)=K(z,w)-\frac{K(z,x)K(x,w)}{K(x,x)},\quad z,w \in \R^d, \label{KPalm}
\end{align} 
whenever $K(x,x)>0$ (see \cite[Theorem 1.7]{ST03}). By \cite[Theorem 3]{G10} (see also \cite{MR21}), the process $\xi^{x!}$ is stochastically dominated by $\xi$ (denoted by $\xi^{x!}\le \xi$) which means that
\begin{align}
	\mathbb{E} [F(\xi^{x!})] \le \mathbb{E} [F(\xi)] \label{Palmdom}
\end{align}
for each measurable $F:\mathbf{N} \to \mathbb{R}$ which is bounded and increasing, by which we mean that $F(\om_1) \le F(\om_2)$ if $\om_1 \subset \om_2$.

For $x \in \R^d$ let $\xi^x$ be a Palm process of $\xi$ at $x$ and $\xi^{x,\Xi}$ a Palm process of $\xi$ with respect to $\Xi$ at $x$. Then we have
\begin{align}
\E \Big[ \int f(x,\xi) \,\Xi(\d x) \Big] =	\E\Big[	\int f(x,\xi) g(x,\xi) \xi(\mathrm{d}x)\Big]&= \int \E [f(x,\xi^x) g(x,\xi^x)] \rho(x) \la(\mathrm{d}x)\nonumber\\
&=\int \E [f(x,\xi^{x,\Xi})] \E [g(x,\xi^x)] \rho(x) \la(\mathrm{d}x). \label{defpalmxi}
\end{align}

An important property of determinantal point processes is that they are negatively associated (NA) (see \cite[Theorem 3.7]{L14}). We say that a point process $\xi$ is NA if for each collection of disjoint sets $B_1,\dots,B_m \in \mathcal{B}^d$ and each subset $I \subset \{1,\dots,m\}$ we have that
\begin{align}
	\Cov(F(\xi(B_i),\,i \in I),G(\xi(B_i),\,i \in I^c))\le 0, \label{defNA}
\end{align}
where $F,G$ are real bounded and increasing functions (see \cite{BW85}).

Let $F:\mathbb{R} \to \mathbb{R}$ be bounded and increasing and assume that $g$ is increasing in the second argument. Then we find for almost all $x \in \R^d$ from \eqref{defpalmxi} and \eqref{defNA} (applied to the determinantal point process $\xi^{x!}$) that
\begin{align}
	\E [F(\xi^{x,\Xi})] \E [g(x,\xi^x)] = \E [F(\xi^{x}) g(x,\xi^{x})] \le \E [F(\xi^{x})] \E [g(x,\xi^{x})], \label{xiPalmdomincr}
\end{align}
implying that $\xi^{x,\Xi} \le \xi^x$ for $\E[\Xi]$-a.a.\ $x \in \R^d$.
On the other hand, if $g$ is decreasing in the second argument, then we find by taking $-g$ in \eqref{defNA} that
\begin{align}
	\E [F(\xi^{x,\Xi})] \E [g(x,\xi^x)] = \E [F(\xi^{x}) g(x,\xi^{x})] \ge \E [F(\xi^{x})] \E [g(x,\xi^{x})], \label{xiPalmdomdecr}
\end{align}
implying that $\xi^{x,\Xi} \ge \xi^x$ for $\E[\Xi]$-a.a.\ $x \in \R^d$.

\subsection{Fast decay of correlation}

Let $\xi$ be a stationary determinantal process on $\R^d$ with covariance kernel $K$ that satisfies the conditions (i)--(iv) and $|K(x,y)|\le \phi(\|x-y\|)$ for some exponentially decreasing function $\phi$ (see \eqref{fastdecay}). Then we have from \cite[Lemma 1.3]{BYY19supp} that the correlation functions $\r^{(m)}$, $m \in \N$, of $\xi$ satisfy
\begin{align}
	|\r^{(p+q)}(x_1,\dots,x_{p+q})- \r^{(p)}(x_1,\dots,x_p)\r^{(q)}(x_{p+1},\dots,x_{p+q})|\le m^{1+\f m2} \phi(s) \|K\|^{m-1},\label{eq:rhodec}
\end{align}
where $m:=p+q$, $s:=d(\{x_1,\dots,x_p\},\{x_{p+1},\dots,x_{p+q}\}):=\inf_{i \in \{1,\dots,p\}, j \in \{p+1,\dots,p+q\}} |x_i-x_j|$, and $\|K\|:=\sup_{x,y\in  \R^d} |K(x,y)|$.

\subsection{Poisson process approximation}

Let $\mathbf L$ be the (finite) intensity measure of the point process $\Xi$ defined at \eqref{def:Xigen}. Suppose for $x \in W$ that $\Xi^x$ is a Palm version of $\Xi$ at $x$ and that $\tilde \Xi \stackrel{d}{=}\Xi$. Then we obtain from \cite[Theorem 3.1]{BSY21} that the Kantorovich-Rubinstein distance of $\Xi$ defined at \eqref{def:Xigen} and a finite Poisson process $\zeta$ with intensity measure $\mathbf M$ is bounded by
\begin{align}
	\mathbf{d_{KR}}(\Xi \cap W,\zeta) \le d_{TV}(\mathbf L (\cdot \cap W),\mathbf M) + 2 \int_{W} \E[(\Xi^x \De \tilde \Xi^x) (W)]\, \mathbf L(\d x).\label{eq:KRbou}
\end{align}


The overall idea of the proofs of Theorem \ref{th:decr} is to apply this inequality with a particular choice of $\Xi_n^x$ and $\tilde \Xi_n^x$. Note that if $\xi^{x,\Xi}$ is a Palm version of $\xi$ at $x$ with respect to $\Xi$, then $\Xi[ \xi^{x,\Xi}] - \delta_x$ is a reduced Palm version of $\Xi$. This allows us to reduce the problem of constructing a coupling of $\Xi$ and its reduced Palm measure to constructing a coupling of $\xi$ and its Palm measure with respect to $\Xi$.

Moreover, for each $x \in W$ we will split the expected symmetric difference $\E[(\Xi_n^x \De \tilde \Xi_n^x) (W)]$ into a part that represents local contributions around $x$ and the rest. This is done as follows. For $x \in W$ let $T_x \in \mathcal B^d$ with $x \in T_x$ and $T_x \subset W$. Then we have for $\Xi_n^{x}:=\Xi[\xi^{x}]$ and $\tilde \Xi^{x}:=\Xi[\hat \xi^{x}]$ that almost surely,
\begin{align*}
(\Xi^{x}\De \tilde \Xi^{x})(W) &\le \Xi[\xi^{x}](T_x) +\Xi[ \xi^{x,\Xi}](T_x) + (\Xi^x \De \tilde \Xi^x)(W \setminus T_x)\\
&\le \xi^{x}(T_x)  + \xi^{x,\Xi}(T_x) + (\Xi^x \De \tilde \Xi^x)(W \setminus T_x). 
\end{align*}

From the particular form of the score functional $\Xi$ we obtain that the last term is almost surely bounded by
$$
 (\Xi^{x} \Delta \tilde \Xi^{x})(W \setminus T_x) \le (\xi^{x} \Delta \xi^{x,\Xi}) (W \setminus T_x) +  \sum_{y \in \xi^{x} \cap  \xi^{x,\Xi}\cap W\setminus T_x} |g(y,\xi^{x})-g(y,\xi^{x,\Xi})|. 
$$

Since we will rely on these observations in the sequel, we summarize them in the following proposition.

\begin{proposition} \label{prop:Poi}
For $x,y \in W$ let $\xi^x$ be a Palm version of $\xi$ at $x$, let $\xi^{x,y}$ be a Palm version of $\xi^x$ at $y$ and let $\xi^{x,\Xi}$ be a Palm version of $\xi$ with respect to $\Xi$ at $x$. Then we have
	\begin{align*}
		\mathbf{d_{KR}}(\Xi \cap W,\zeta) \le d_{TV}(\mathbf L (\cdot \cap W),\mathbf M) +T_1+T_2+T_3+T_4,
	\end{align*}
where
\begin{align*}
T_1:=&\r^2 \int_W \int_{T_x} \E[g(x,\xi^{x})] \E[g(y,\xi^y)] \d y \d x,\\
T_2:=&\r^2 \int_W \int_{T_x} \E[g(x,\xi^{x,y})g(y,\xi^{x,y})] \d y \d x,\\
T_3:=&\r \int_W\E[(\xi^{x} \Delta \xi^{x,\Xi}) (W \setminus T_x)] \d x,\\
T_4:=&\r \int_W  \E\Big[  \sum_{y \in \xi^{x} \cap  \xi^{x,\Xi}\cap W\setminus T_x} |g(y,\xi^{x})-g(y,\xi^{x,\Xi})|\Big] \d x.
\end{align*}
\end{proposition}

\section{Proof of Theorem \ref{th:decr}}\label{sec:pr1}

\begin{proof}[Proof of Theorem \ref{th:decr}]
For $x \in W$ let $S_x:=x+S$ and $T_x:=x+T$. We let $\tilde g(x,\om):=g(x,\om) \one\{\mathcal S(x,\om)\subset S_x\}$, $\om \in \mathbf N$, and consider the truncated functional
	\begin{align}
		\Xi_{\ms{tr}}:=\sum_{x \in \xi \cap W} \tilde g(x,\xi) \delta_{x}. \label{def:Xitr}
	\end{align}
	Note that, due to the stopping set property of $\mathcal S(x,\xi)$, $\tilde g(x,\xi)$ is measurable with respect to $\xi \cap S_x$.
	
	{\em Step 1:} In a first step we assume that $g=\tilde g$ and, therefore, $\Xi=\Xi_{\ms{tr}}$. Let $\mathbf L$ be the intensity measure of $\Xi$.	We apply Proposition \ref{prop:Poi} and bound the terms $T_3$ and $T_4$.
	
	{\bf Bounding $T_3$.}
For each $x \in W$, we construct a coupling $(\xi',\xi'')$ of $P^{x!}$ and $P^{x!,\Xi}$ such that the symmetric difference $(\xi' \Delta \xi'') (W\setminus T_x)$ becomes small. We discuss this for increasing and decreasing scores separately. 

(i) {\em Increasing scores.} If $g(x,\om_1)\le g(x,\om_2)$ for $\om_1 \subset \om_2$, we have by \eqref{Palmdom} and \eqref{xiPalmdomincr} that $\xi^{x!} \le \xi$ and $\xi^{x!,\Xi} \le \xi^{x!}$, implying that $\xi^{x!,\Xi}  \le \xi$. By Strassen's theorem, there exists a Palm version $\xi^{x,\Xi}$ of $\xi$ with respect to $\Xi$ and a process $\tilde \xi^x\stackrel{d}{=}\xi$ such that $\xi^{x,\Xi} \subset \tilde \xi^x$ almost surely. Thus, we have
	\begin{align}
		\E [(\xi^{x!,\Xi} \Delta \tilde \xi^x) (W\setminus T_x)]&=\E[\xi(W\setminus T_x)] -\E[\xi^{x!,\Xi}(W\setminus T_x)]\nonumber\\
		&=\big\{\E [\xi (W\setminus T_x)]- \E[ \xi^{x!}(W\setminus T_x)]\big\} +\E[\xi^{x!}(W\setminus T_x)]-\E[ \xi^{x!,\Xi}(W\setminus T_x)]. \label{etabou1}
	\end{align}

	The term in $\{\cdots\}$ on the right-hand side above is by  \eqref{KPalm}, the fact that $K(x,x)=\r$ for all $x \in \R^d$, \eqref{fastdecay} and since $\phi$ is decreasing, given by
	\begin{align}
\r^{-1}\int_{W \setminus T_x} |K(x,y)|^2 \d y \le \r^{-1}\int_{W\setminus T_x}  \phi(\|x-y\|)^2  \d y \le \r^{-1} |W| \sup_{y \in T}\phi(\|y\|)^2,\quad x \in W. \label{MRdens}
	\end{align}
	Next we consider the second term on the right-hand side in \eqref{etabou1}. By definition of the reduced Palm process $\xi^{x!,\Xi}$, we have for almost all $x\in W$,
	\begin{align}
		\E[g(x,\xi^{x})]\big\{\E[\xi^{x!}(W\setminus T_x)]-\E[ \xi^{x!,\Xi}(W\setminus T_x)]\big\}&=\E [\xi^{x} (W\setminus T_x)]\E [\tilde g(x,\xi^x)]- \E[\xi^x(W \setminus T_x) \tilde g(x,\xi^x)]\nonumber\\
		&=-\Cov(\xi^{x} (W\setminus T_x), \tilde g(x,\xi^{x})). \label{covbou1}
	\end{align}
	Now we use that the reduced Palm process $\xi^{x!}$ is a determinantal process itself and therefore negatively associated (see \eqref{defNA}). For $k \in \mathbb{N}$ we consider the auxiliary functions
	\begin{align*}
		f^{(k)}(\om):=\min\{k,\om(S_x)-\tilde g(x,\om)\},\quad f(\om):=\om(S_x)-\tilde g(x,\om),\quad \om\in \mathbf{N}.
	\end{align*}
	It is easy to see that $f^{(k)},\, k \in \N,$ and $f$ are bounded and increasing. Since $\xi^{x!}$ is negatively associated, we have that
	\begin{align*}
		\Cov(\min\{k,\xi^{x!}(W\setminus T_x)\}, f^{(k)}(\xi^{x!}))\le 0. 
	\end{align*}
	Hence, by monotone convergence,
	\begin{align*}
	&\Cov(\xi^{x!}(W\setminus T_x),\xi^{x!}(S_x))-\Cov(\xi^{x!}(W\setminus T_x), \tilde g(x,\xi^{x!}))\\
	&=\Cov(\xi^{x!}(W\setminus T_x), f(\xi^{x!}))=\lim_{k \to \infty} \Cov(\min\{k,\xi^{x!}(W\sm T_x)\}, f^{(k)}(\xi^{x!}))\le 0.
\end{align*}
	This shows that \eqref{covbou1} is bounded by
	\begin{align}
		&-\Cov(\xi^{x!}(W\setminus T_x),\xi^{x!}(S_x))\nonumber\\
		&\quad=\big(\E [\xi^{x!} (W\setminus T_x)] - \E[\xi(W \setminus T_x)]\big) \E[\xi^{x!}(S_x)]-\big(\E[\xi^{x!}(W\setminus T_x)\xi^{x!}(S_x)] - \E[\xi(W\setminus T_x)]\E[\xi^{x!}(S_x)]\big)\nonumber\\
		&\quad \le \big|\E [\xi^{x!} (W\setminus T_x)] - \E[\xi(W \setminus T_x)]\big| \E[\xi^{x!}(S_x)]+\big|\E[\xi^{x!}(W\setminus T_x)\xi^{x!}(S_x)] - \E[\xi(W\setminus T_x)]\E[\xi^{x!}(S_x)]\big|.
	\label{covxixi1}
	\end{align}
Here, we use \eqref{MRdens} and $\E[\xi^{x!}(S_x)]\le \E[\xi(S_x)]= \r |S|$ to bound the first term on the right-hand by
$$
\big|\E [\xi^{x!} (W\setminus T_x)] - \E[\xi(W \setminus T_x)]\big| \E[\xi^{x!}(S_x)] \le |W| |S| \sup_{y \in T}\phi(\|y\|)^2.
$$
Next we consider the second term on the right-hand side of \eqref{covxixi1}. We write $\r_m^{x}$ for the $m$-th correlation function of $\xi^{x!}$, we obtain by \cite[Lemma 6.4]{ST03}, by the definition of $\xi^{x!}$ and by \eqref{eq:rhodec} that
\begin{align*}
\r\int_W \Big(\E[\xi^{x!}(W\setminus T_x)\xi^{x!}(S_x)] - \E[\xi(W\setminus T_x)]\E[\xi^{x!}(S_x)]\Big) \d x&=\r \int_W	\int_{S_x} \int_{W\setminus T_x} (\r_2^{x}(y,z) -  \r^{x}(y) \r) \,\d z \d y \d x\\ &= \int_{W} \int_{S_x} \int_{W\setminus T_x} (\r_3(x,y,z) -  \r_2(x,y) \r) \,\d z \d y \d x\\
	&\le3^{5/2} \|K\|^2 \int_{W} \int_{S_x} \int_{W\setminus T_x} \hspace{-0.3cm} \phi(d(\{x,y\},\{z\}))\, \d z \d y \d x\\
	&\le 3^{5/2} \|K\| |W|^2 |S| \phi(d(T,S)).
\end{align*}
Thus, we can conclude that 
\begin{align}
T_3=\r \int_W \E [(\xi^{x!,\Xi} \Delta \tilde \xi^x) (W\setminus T_x)] \d x&\le (1+\r |S|) |W|^2\sup_{y \in T}\phi(\|y\|)^2+3^{5/2} \|K\| |W|^2 |S| \phi(d(T,S)).\label{symbou}
\end{align}

	(ii) {\em Decreasing scores.} If $g(x,\om_1)\ge g(x,\om_2)$ for $\om_1 \subset \om_2$, we have by \eqref{Palmdom} and \eqref{xiPalmdomdecr} that $\xi^{x!} \le \xi$ and $\xi^{x!} \le \xi^{x!,\Xi}$. Let $\xi^{x}$ be a Palm version of $\xi$ at $x$. By Strassen's theorem and \cite[Theorem 2.15]{K02}, there exists a point process $\tilde \xi^{x} \stackrel{d}{=}\xi$ and a Palm process $\xi^{x,\Xi}$ of $\xi$ with respect to $\Xi$ such that $\xi^{x!} \subset \tilde \xi^x$ and $\xi^{x!}\subset \xi^{x!,\Xi}$. This gives 
\begin{align}
	\E [(\xi^{x,\Xi} \Delta \tilde \xi^x) (W\setminus T_x)]&\le \E[(\tilde \xi^x\setminus \xi^{x!})(W\setminus T_x)] +\E[(\xi^{x,\Xi} \setminus \xi^{x!})(W\setminus T_x)] \nonumber\\
	&= \big\{\E [\xi (W\setminus T_x)]- \E[ \xi^{x!}(W\setminus T_x)]\big\} +\big\{\E[ \xi^{x!,\Xi}(W\setminus T_x)]-\E[\xi^{x!}(W\setminus T_x)] \big\}.\label{etabou2}
\end{align}

	Here we bound the term in $\{\cdots\}$ as in \eqref{MRdens}. For the second term in \eqref{etabou2} we obtain
		\begin{align}
		\E[g(x,\xi^{x})]\big\{\E[ \xi^{x!,\Xi}(W\setminus T_x)]-\E[\xi^{x!}(W\setminus T_x)]\big\}&=\E[\xi^x(W \setminus T_x) \tilde g(x,\xi^x)]-\E [\xi^{x} (W\setminus T_x)]\E [\tilde g(x,\xi^x)]\nonumber\\
		&=\Cov(\xi^{x} (W\setminus T_x), \tilde g(x,\xi^{x})). \label{covbou2}
	\end{align}
	Now we use that the reduced Palm process $\xi^{x!}$ is a determinantal process itself and therefore negatively associated (see \eqref{defNA}). For $k \in \mathbb{N}$ we consider the auxiliary functions
	\begin{align*}
		f^{(k)}(\om):=\min\{k,\om(S_x)+\tilde g(x,\om)\},\quad f(\om):=\om(S_x)+\tilde g(x,\om),\quad \om\in \mathbf{N}.
	\end{align*}
	It is easy to see that $f^{(k)},\, k \in \N,$ and $f$ are bounded and increasing. Since $\xi^{x!}$ is negatively associated, we have that
	\begin{align*}
		\Cov(\min\{k,\xi^{x!}(W\setminus T_x)\}, f^{(k)}(\xi^{x!}))\le 0. 
	\end{align*}
	Hence, by monotone convergence,
	\begin{align*}
		\Cov(\xi^{x!}(W\setminus T_x),\xi^{x!}(S_x))+\Cov(\xi^{x!}(W\setminus T_x), \tilde g(x,\xi^{x!}))&=\Cov(\xi^{x!}(W\setminus T_x), f(\xi^{x!}))\\
		&=\lim_{k \to \infty} \Cov(\min\{k,\xi^{x!}(W\sm T_x)\}, f^{(k)}(\xi^{x!}))\le 0.
	\end{align*}
	This shows that \eqref{covbou2} is bounded by $-\Cov(\xi^{x!}(W\setminus T_x),\xi^{x!}(S_x))$. Therefore, we can process as for increasing scores and obtain the same bound for $T_3$ as in \eqref{symbou}.
	
	 
	{\bf Bounding $T_4$.} 
For each $x \in W$ we have
	\begin{align*}
		&\E \Big[\sum_{y \in \xi^{x!,\Xi} \cap \tilde \xi^x \cap W\setminus T_x} |\tilde g(y,\xi^{x!,\Xi})-\tilde g(y,\tilde \xi^x)|\Big]\\
		&\quad =	\E \Big[\sum_{y \in \xi^{x!,\Xi} \cap \tilde \xi^x \cap W\setminus T_x} \one\{(\xi^{x!,\Xi} \Delta \tilde \xi^x)\cap S_x \neq \es\} |\tilde g(y,\xi^{x!,\Xi})-\tilde g(y,\tilde \xi^x)|\Big]\\
		&\quad \le \E \Big[\sum_{z \in (\xi^{x!,\Xi} \Delta \tilde \xi^x) \cap (W\oplus S_x) \setminus T_x} \sum_{y \in \xi^{x!,\Xi} \cap \tilde \xi^x  \cap S_z}  |\tilde g(y,\xi^{x!,\Xi})-\tilde g(y,\tilde \xi^x)|\Big]\\
		&\quad \le \E \Big[\sum_{z \in (\xi^{x!,\Xi} \Delta \tilde \xi^x) \cap (W\oplus S_x) \setminus T_x} \max_{\om \in \{\xi^{x!,\Xi},\tilde \xi^x\}}\sum_{y \in \om \cap S_z}  \tilde g(y,\om)\Big].
	\end{align*}

Here we obtain from Condition \eqref{kiss} that the above is bounded by
		$$
		 \alpha |S|  \E \Big[(\xi^{x!,\Xi} \Delta \tilde \xi^x) ((W\oplus S_x) \setminus T_x)\Big]. 
$$
	Hence, we obtain from the estimate in \eqref{symbou} (with $W$ replaced by $W \oplus S_x$) that
	\begin{align}
		T_4&=\r \int_W 	\E \Big[\sum_{x \in \xi^{x!,\Xi} \cap \tilde \xi^x \cap W\setminus T_x} |\tilde g(x,\xi^{x!,\Xi})-\tilde g(x,\tilde \xi^x)|\Big] \d x\nonumber\\
		&  \le(1+\r \sup_{x \in W}|S|) |W\oplus S|^2 \sup_{y \in T} \phi(\|y\|)^2+3^{5/2} \|K\| |W\oplus S|^2 |S| \phi(d(T,S)).
		\label{xibou2}
	\end{align}
	
	{\em Step 2:} Finally, we shall complete the proof in a second step. Let  $\mathbf L$ and $\mathbf L_{\ms{tr}}$ denote the intensity measure of the truncated process $\Xi_{\ms{tr}}$ from \eqref{def:Xitr} and of $\Xi$, respectively. By definition of the Palm process $\xi^x$, we have
	\begin{align*}
		\mathbf L_{\ms{tr}} (A)=\r \int_{A \cap W} \E [g(x,\xi^x) \one \{\mathcal S(x,\xi^x)\subset S_x\}]\, \d x,\quad A \in \mathcal B^d.
	\end{align*}
	Let $\zeta_{\ms{tr}}$ be a Poisson process with intensity measure $\mathbf L_{\ms{tr}}$. From the triangle inequality for the KR distance we find that
	\begin{align*}
		\mathbf{d_{KR}}(\Xi,\zeta)\le \mathbf{d_{KR}}(\Xi,\Xi_{\ms{tr}}) + \mathbf{d_{KR}}(\Xi_{\ms{tr}}, \zeta_{\ms{tr}})+\mathbf{d_{KR}}(\zeta_{\ms{tr}},\zeta).
	\end{align*}
	From Section 3 we know that $\mathbf{d_{KR}}(\zeta_{\ms{tr}},\zeta) \le d_{TV}(\mathbf L_{\ms{tr}}, \mathbf L)$ and $\mathbf{d_{KR}}(\Xi,\Xi_{\ms{tr}}) \le \mathbf{d_{TV}}(\Xi,\Xi_{\ms{tr}})$, where
	\begin{align*}
		\mathbf{d_{TV}}(\Xi,\Xi_{\ms{tr}})= \mathbb{E} \Xi(W)-\mathbb{E} \Xi_{\ms{tr}}(W)=d_{TV}(\mathbf L, \mathbf L_{\ms{tr}}) 
	\end{align*}
	and
	\begin{align*}
		d_{TV}(\mathbf L, \mathbf L_{\ms{tr}})  =\r \int_W \mathbb{E}[g(x,\xi^x) \one\{\mathcal S(x,\xi^x) \not \subset S_x\}]\, \d x.
	\end{align*}
	This shows that
	\begin{align}
		\mathbf{d_{KR}}(\Xi,\zeta)&\le \mathbf{d_{KR}}(\Xi_{\ms{tr}},\zeta_{\ms{tr}})+2\int_W \mathbb{E}[g(x,\xi^x) \one\{\mathcal S(x,\xi^x) \not \subset S_x\}]\, \d x.\label{step2}
	\end{align}
	Combining $\eqref{step2}$ with \eqref{symbou} and \eqref{xibou2} from {\em Step 1} gives the assertion.
\end{proof}

\section{Proof of Theorem \ref{th:balls}} \label{sec:pr2}

In the proof of Theorem \ref{th:balls}, we repeatedly use that the set of absolute values of the points of the (infinite) Ginibre process $\xi$ has the same distribution as a sequence $(X_i)_{i \in \N}$ of independent random variables with $X_i^2 \sim \text{Gamma}(i,1)$ (see \cite{K92} or \cite[Theorem 26]{BKPV06}). This implies that the mapping $r \mapsto \P(\xi(B_r)=0)$ is continuous and that $\P(\xi(B_r)=0) \downarrow 0$ as $r \to \infty$. Hence, for all $\tau>0$ there exists an increasing sequence $(v_n)_{n \in \N}$ such that 
\begin{equation}
\mathbf L_n(A):=\mathbb{E} [\Xi_n(A)]= \f{|A \cap B_n|}{\pi} \mathbb{P}(\xi^{o!}(B_{v_n})=0)=\f{\tau |A \cap B_n|}{\pi n^2},\quad n \in \mathbb N,\, A \in \mathcal B^2.	\label{nnintass}
\end{equation}

To determine the asymptotic behavior of $v_n$ as $n \to \infty$, we use that by \cite[Theorem 26]{BKPV06},
\begin{equation}
	\mathbb{P}(\xi^{o!}(B_{v_n})=0)=e^{v_n^2}\mathbb{P}(\xi(B_{v_n})=0).	\label{nnvoidprob}
\end{equation}
Moreover, by \cite[Proposition 7.2.1]{BKPV09},
\begin{equation*}
	\lim_{r \to \infty} \frac{1}{r^4} \log \mathbb{P}(\xi(B_{r})=0)=-\frac 14.
\end{equation*}
Therefore, re-writing $\log \mathbb{P}(\xi(B_{r})=0)$ as
$$
\log (n e^{v_n^2}\mathbb{P}(\xi(B_{r}=0))-\log(n) v_n^2,
$$
we find from \eqref{nnintass} and \eqref{nnvoidprob} that $\frac{v_n^4}{\log n}\to 4$ as $n \to \infty$. 

\begin{proof}[Proof of Theorem \ref{th:balls}]
Given $n \in \N$, we choose $\xi$ as the Ginibre process, let $g(x,\om):=\one\{\om(B_{v_n}(x)\sm \{x\})=0\}$, $S_x:=B_{v_n}(x)$ and $T_x:=B_{\log n}(x)$, $x \in \R^2$. Note that $g$ is stabilizing with respect to the (deterministic) stopping set $\mathcal S(x,\om):= S_x,\, \om \in \mathbf N$. 
The idea is now to apply Theorem \ref{th:decr} to the process $\Xi_n$ defined at \eqref{nnxidef}, where we choose $\zeta$ is a stationary Poisson process with intensity $\f{\tau}{\pi n^2}$. Then, $\zeta \cap B_n$ has intensity measure $\mathbf L_n$ given at \eqref{nnintass}. Since $g$ is deterministically stabilizing, we have that $E_1=0$. Note moreover that $E_2$ and $F$ are bounded by constant multiples of $\f{(\log n)^2}{n}$. Thus, by the contractivity property of the KR distance and by Theorem \ref{th:decr}, we have
\begin{align}
	\mathbf{d_{KR}}(\Psi_n,\nu \cap B_1) &\le 	\mathbf{d_{KR}}(\Xi_n,\zeta \cap B_n)\nonumber\\
	&\le \int_{B_n}\int_{T_x} \P(\xi^{x!,y!}(B_{v_n}(x)\cup B_{v_n}(y))=0)\rho_2(x,y) \d y \d x + \beta \f{(\log n)^2}{n} \label{bou:krballs}
\end{align}
for some $\beta>0$. Thus, it remains to bound the integral
on the right-hand side above. We first note that by \cite[Theorem 6.5]{ST03}, the reduced Palm process $\xi^{x!,y!}$ is a determinantal process itself. Hence, we can conclude from \cite[Theorem 3.7]{L14} (see also \cite[Theorem 3.2]{LS19}) that $\xi^{x!,y!}$ has negative associations. Recall that a point process $\zeta$ has negative associations if $\E[f(\zeta)g(\zeta)]\le \E[f(\zeta)] \E[g(\zeta)]$ for every pair $f,g$ of real bounded increasing (or decreasing) functions that are measurable with respect to complementary subset $A, A^c$ of $\R^d$, meaning that a function is measurable with respect to $A$ if it is measurable with respect to $\mathcal N_A$. We apply this with the decreasing functions $f(\mu)=\one\{\mu (B_{v_n}(x))=0\}$ and $g(\mu)=\one\{\mu (B_{v_n}(y) \setminus B_{v_n}(x))=0\}$. This gives 
	\begin{equation}
		\P(\xi^{x!,y!}(B_{v_n}(x) \cup B_{v_n}(y))=0)\leq 	\P(\xi^{x!,y!}(B_{v_n}(x))=0)\, \P(\xi^{x!,y!}(B_{v_n}(y)\setminus B_{v_n}(x))=0). \label{T2bou}
	\end{equation}	
	To bound the first probability, we note that by \cite[Theorem 1]{G10} there is a coupling  $(\xi',\xi'')$ of $(\xi^{x!},\xi^{x!,y!})$ such that $\xi''\subset \xi'$ and $|\xi' \setminus \xi''| \le 1$ a.s. This gives 
	$$
	\P(\xi^{x!,y!}(B_{v_n}(x))=0) \le 	\P(\xi^{x!}(B_{v_n}(x)) \le 1).
	$$
	Now we apply the same argument with a coupling of $(\xi, \xi^{x!})$ and obtain the bound
	$$
	\P(\xi(B_{v_n}(x))\le 2)=\P(\xi(B_{v_n})\le 2),
	$$
	where the last equality holds due to the stationarity of $\xi$. As mentioned at beginning of this section, the set of absolute values of the points of the (infinite) Ginibre process $\xi$ has the same distribution as a sequence $(X_i)_{i \in N}$ of independent random variables with $X_i^2 \sim \text{Gamma}(i,1)$. This gives 
	\begin{align*}
		\P(\xi(B_{v_n}) \le 2) &=\P(\#\{j \in \N:\,X_j \le v_n\} \le 2)\\ 
		& \le \P(\#\{j \in \{1,\dots,v_n^2\}:\,X_j \le v_n\} \le 2)\\
		&=\P\Big(\bigcup_{i=1}^{v_n^2} \bigcup_{\substack{j=1\\ j \neq i}}^{v_n^2} \{ \forall k \in \{1,\dots,v_n^2\}\setminus \{i,j\}:\,X_k > v_n \}\Big).
	\end{align*}
	In the above equation, with a slight abuse of notation, we have written $v_n^2$ instead of $\lfloor v_n^2\rfloor$. The union bound yields that the above is bounded by
	\begin{align*}
		\sum_{i=1}^{v_n^2} \sum_{\substack{j=1\\ j \neq i}}^{v_n^2} \P(\forall k \in \{1,\dots,v_n^2\}\setminus \{i,j\}:\,X_k > v_n)
		=\sum_{i=1}^{v_n^2} \sum_{\substack{j=1\\ j \neq i}}^{v_n^2} \prod_{\substack{k =1\\ k\neq i,j}}^{v_n^2} \mathbb P(X_k^2 >v_n^2).
	\end{align*}
	Let $t<1$. The moment generating function $M_{X_k^2}(t)=\E[e^{tX_k^2}]=(1-t)^{-k}$ of $X_k^2$  exists and we obtain from the Chernoff bound that
	$$
	\P(X_k^2>r^2) \le e^{-t r^2} \E [e^{t X_k^2}]=e^{-t r^2} (1-t)^{-k}.
	$$
	For $k<r^2$, this bound is maximized for $t=1- \frac{k}{r^2}$, which gives
	$$
	\P(\xi(B_{v_n}) \le 2)\le \sum_{i=1}^{v_n^2} \sum_{\substack{j=1\\ j \neq i}}^{v_n^2} \prod_{\substack{k =1\\ k\neq i,j}}^{v_n^2} e^{-(1-\frac {k}{v_n^2})v_n^2-k \log\big(\frac{k}{v_n^2}\big)}=\sum_{i=1}^{v_n^2} \sum_{\substack{j=1\\ j \neq i}}^{v_n^2} \prod_{\substack{k =1\\ k\neq i,j}}^{v_n^2} e^{-v_n^2+k-k \log \big(\frac{k}{v_n^2}\big)}.
	$$
	Using here that $u \mapsto u-u\log (u/r^2)$ is increasing  for $u\le r^2$, we find that 
	\begin{align*}
		\P(X_k^2>r^2) & \leq \sum_{i=1}^{v_n^2} \sum_{\substack{j=1\\ j \neq i}}^{v_n^2}\prod_{k=3}^{v_n^2} e^{-v_n^2+k-k \log \big(\frac{k}{v_n^2}\big)}\\
		&\le v_n^4 \prod_{k=3}^{v_n^2} e^{-v_n^2+k-k \log \big(\frac{k}{v_n^2}\big)}\\
		&=v_n^2 e^{-\frac12 (v_n^2-3)(v_n^2-2)-v_n^4\int_{3/v_n^2}^{1}x \log (x) \mathrm{d}x+O(v_n^2 \log v_n)}\\
		&=e^{-\frac14 v_n^4(1+o(1))}
	\end{align*}
	as $n \to \infty$, where we have used that $\int_0^1 x \log (x) \d x=-\frac 14$.
	
	Next we bound the second probability in \eqref{T2bou}. By the same coupling argument as above we find that
	$$
	\mathbb{P}(\xi^{x!,y!}(B_{v_n}(y)\sm B_{v_n}(x))=0)\le \P(\xi (B_{v_n}(y)\sm B_{v_n}(x))\le 2).
	$$
	Next we note that $B_{v_n/2}\left(y+\frac{v_n(y-x)}{2|y-x|}\right) \subset B_{v_n}(y)\setminus B_{v_n}(x)$. Hence, the last probability is bounded by
	$$
	\P\Big(\xi\big(B_{v_n/2}\big(y+\frac{v_n(y-x)}{2|y-x|}\big)\big)\le 2\Big)=\P(\xi(B_{v_n/2})\le 2) \le e^{-\frac 14 (v_n/2)^4 (1+o(1))}
	$$
	by the same estimates as above (with $v_n/2$ instead of $v_n$). Since  $\rho_2(x,y)\le 1/\pi^2$ for all $x,y \in \R^2$, we arrive for all $\eps>0$ at the bound 
	\begin{align*}
\int_{B_n}\int_{T_x} \P(\xi^{x!,y!}(B_{v_n}(x)\cup B_{v_n}(y))=0)\rho_2(x,y) \d(x,y)\le \frac{n (\log n)^2}{\pi} e^{-\frac 14 v_n^4}  e^{-\frac 1{64} v_n^4} \le \f{(\log n)^2}{\pi} n^{\eps-1/16},
	\end{align*}
where we have used \eqref{nnintass} and that $\frac{v_n^4}{\log n}\to 4$ as $n \to \infty$. Hence, the assertion follows from \eqref{bou:krballs}.
\end{proof}

\end{document}